\documentclass[11pt]{article}

\usepackage{multirow}
\usepackage{subfig}
\usepackage{pythonhighlight}
\usepackage{algorithm}
\usepackage{algpseudocode}

\algdef{SE}
[CLASS]
{Class}
{EndClass}
[1]
{\textbf{class} \textsc{#1}}
{\textbf{end class}}

\algdef{SE}
[FUNCTION]
{Function}
{EndFunction}
[1]
{\textbf{function} \textsc{#1}}
{\textbf{end function}}

\usepackage{xcolor}
\def\HiLi{\leavevmode\rlap{\hbox to \linewidth{\color{yellow!50}\leaders\hrule height .8\baselineskip depth .5ex\hfill}}}

\usepackage{float}
\usepackage[utf8]{inputenc}
\usepackage[T1]{fontenc}
\usepackage{amsmath,amssymb}
\usepackage{graphicx}
\usepackage{float}
\numberwithin{equation}{section}
\usepackage[hyperref,amsmath,thmmarks]{ntheorem}
\usepackage{aliascnt}
\usepackage[a4paper,centering,bindingoffset=0cm,marginpar=2cm,margin=2.5cm]{geometry}
\usepackage[pagestyles]{titlesec}
\usepackage[font=footnotesize,format=plain,labelfont=sc,textfont=sl,width=0.75\textwidth,labelsep=period]{caption}
\usepackage[backend=biber,maxnames=10,backref=true,hyperref=true,giveninits=true,safeinputenc]{biblatex}
\usepackage{color}
\definecolor{teal}{rgb}{0.0, 0.5, 0.5}

\DefineBibliographyStrings{english}{%
	backrefpage = {cited on page},
	backrefpages = {cited on pages},
}

\title{Approximation properties of shallow quadratic neural networks and clustering applications}
\author{Leon Frischauf$^1$\\{\footnotesize\href{mailto:leon.frischauf@univie.ac.at}{leon.frischauf@univie.ac.at}}
\and Otmar Scherzer$^{1,2,3}$\\{\footnotesize\href{mailto:otmar.scherzer@univie.ac.at}{otmar.scherzer@univie.ac.at}}
\and Cong Shi$^{1,4*}$\\{\footnotesize\href{mailto:shic3@mail.sysu.edu.cn}{cong.shi@univie.ac.at}}}
\date{}

\DeclareFieldFormat[report]{title}{``#1''}
\DeclareFieldFormat[book]{title}{``#1''}
\AtEveryBibitem{\clearfield{url}}
\AtEveryBibitem{\clearfield{note}}

\titleformat{\section}[block]{\large\sc\filcenter}{\thesection.}{0.5ex}{}[]
\titleformat{\subsection}[runin]{\bf}{\thesubsection.}{0.5ex}{}[.]

\usepackage[pdftex,colorlinks=true,linkcolor=blue,citecolor=green,urlcolor=blue,bookmarks=true,bookmarksnumbered=true]{hyperref}
\hypersetup
{
    pdfauthor={Authors},
    pdfsubject={Subject},
    pdftitle={Title},
    pdfkeywords={Keywords}
}

\newpagestyle{headers}
{
	\headrule
	\sethead[\footnotesize\thepage][\footnotesize\sc Authors][]{}{\footnotesize\sc Universal approximation properties of shallow quadratic neural networks}{\footnotesize\thepage}
	\setfoot{}{}{}
}
\newtheorem{lemma}{Lemma}[section]

\newaliascnt{proposition}{lemma}

\aliascntresetthe{proposition}

\newaliascnt{corollary}{lemma}
\newtheorem{corollary}[corollary]{Corollary}
\aliascntresetthe{corollary}

\newaliascnt{theorem}{lemma}
\newtheorem{theorem}[theorem]{Theorem}
\aliascntresetthe{theorem}

\theorembodyfont{\normalfont}
\newaliascnt{definition}{lemma}
\newtheorem{definition}[definition]{Definition}
\aliascntresetthe{definition}

\theorembodyfont{\normalfont}
\newaliascnt{example}{lemma}
\newtheorem{example}[example]{Example}
\aliascntresetthe{example}

\newaliascnt{assumption}{lemma}

\aliascntresetthe{assumption}

\newaliascnt{notation}{lemma}
\newtheorem{notation}[notation]{Notation}
\aliascntresetthe{notation}

\newtheorem{remark}{Remark}

\theoremsymbol{\ensuremath{\square}}
\newtheorem*{proof}{Proof}

\newcommand{\N}{\mathbb{N}}
\newcommand{\Z}{\mathbb{Z}}
\newcommand{\R}{\mathbb{R}}

\let\RE\Re
\let\Re=\undefined
\DeclareMathOperator{\Re}{\RE e}
\let\IM\Im
\let\Im=\undefined
\DeclareMathOperator{\Im}{\IM m}

\newcommand{\abs}[1]{\left|#1\right|}
\newcommand{\norm}[1]{\left\|#1\right\|}
\newcommand{\set}[1]{\left\{ #1\right\}}

\newcommand{\e}{\mathrm e}
\let\ii\i
\renewcommand{\i}{\mathrm i}

\newcommand{\vx}{{\vec{x}}}
\newcommand{\vy}{{\vec{y}}}
\newcommand{\vw}{\vec{w}}
\newcommand{\vz}{\vec{z}}

\newcommand{\vxm}{{\bf{x}}}

\newcommand{\vwm}{{\bf{w}}}
\newcommand{\vfm}{{\bf{f}}}
\newcommand{\vzm}{{\bf{z}}}

\newcommand{\interval}{\texttt{I}}
\newcommand{\quader}{\texttt{Q}}

\newcommand{\I}[1]{\mathcal{I}_{#1}}

\newcommand\blfootnote[1]{%
  \begingroup
  \renewcommand\thefootnote{}\footnote{#1}%
  \addtocounter{footnote}{-1}%
  \endgroup
}

\begin{document}

\maketitle
\thispagestyle{empty}
\begin{center}
  \hspace*{2em}
  \parbox[t]{15em}{\footnotesize\raggedright
    \noindent
    $^1$Faculty of Mathematics\\
    University of Vienna\\
    Oskar-Morgenstern-Platz 1\\
    A-1090 Vienna, Austria\\\vspace*{1ex}
  }
  \hfil
  \parbox[t]{15em}{\footnotesize\raggedright
    \noindent
    $^2$Johann Radon Institute for Computational and Applied Mathematics (RICAM)\\
    Altenbergerstraße 69\\
    A-4040 Linz, Austria
  }
  \hfil
\end{center}
\begin{center}
  \hspace*{2em}
  \parbox[t]{15em}{\footnotesize\raggedright
    \noindent
    $^3$Christian Doppler Laboratory
    for Mathematical Modeling and Simulation
    of Next Generations of Ultrasound Devices (MaMSi)\\
    Oskar-Morgenstern-Platz 1\\
    A-1090 Vienna, Austria
  }
  \hfil
  \parbox[t]{15em}{\footnotesize\raggedright
    \noindent
    $^4$School of Mathematics (Zhuhai)\\
    Sun Yat-Sen University\\
    Hanlin Rd, 519082 Zhuhai\\
    Guangdong Province, China\\\vspace*{1ex}
  }
  \hfil
\end{center}
\blfootnote{${}^*$Cong Shi is the corresponding author.}

\begin{abstract}
In this paper we study \emph{shallow} neural network functions which are linear combinations of compositions of activation and \emph{quadratic} functions, replacing standard
\emph{affine linear} functions, often called neurons.
We show the universality of this approximation and prove convergence rates results based on the theory of wavelets and statistical learning.
We show for simple test cases that this ansatz requires a smaller numbers of neurons than standard affine linear neural networks. Moreover, we investigate the efficiency of this approach for clustering tasks with the MNIST data set.
Similar observations are made when comparing \emph{deep (multi-layer)} networks.
\end{abstract}
\hspace*{5.8ex}\textbf{MSC:} 41A30, 65XX, 68TXX\\
\hspace*{5.8ex}\textbf{Keywords:} Generalized neural network; universal approximation; convergence rates; numerical implementation and algorithm


\section{Introduction}
Approximation of functions with \emph{shallow (single-layer) neural networks} is a classical topic
of machine learning and in approximation theory. The basic mathematical problem
consists in approximating a function $g: \R^n \to \R$ by \emph{neural network
functions} of the form
\begin{equation}\label{eq:classical_approximation}
G(\vx) := \sum_{j=1}^{N} \alpha_j\sigma\left(p_j(\vx) \right) \text{ where }
p_j(\vx) = \vwm_j^T \vx +\theta_j
\text{ with } \alpha_j, \theta_j \in \R \text{ and } \vx, \vwm_j \in \R^{n}.
\end{equation}
Here $\sigma: \R \to \R$ is a given function, called the \emph{activation function}
and $\vwm_j \in \R^n$, $\theta_j \in \R$ and $\alpha_j \in \R$, $j=1,\ldots,N$ are parameters.
We name functions of the form in \autoref{eq:classical_approximation} {\emph affine linear neural networks}.
This approximation problem has been well studied in the literature already in the 80ties and 90ties, see for instance \cite{Pin99,Bar93,Cyb89,HorStiWhi89b,LesLinPinScho93,Mha93}, leading to the \emph{universal approximation property} of affine linear neural networks. Later on the universal approximation property has been established for different classes of neural networks: Examples are dropout neural networks (see \cite{WagWanLia13,ManPelPorSanSen22}), convolutional neural networks (CNN) (see for example \cite{Zho18, Zho20}), recurrent neural networks (RNN) (see \cite{SchZim07,Ham00}), networks with random nodes
(see \cite{Whi89}), with random weights and biases (see \cite{PaoParSob94,IgePao95}) and with fixed neural network topology (see \cite{GelMaoLi99}).

Two classes of neural network are of particular importance for this work: In \cite{TsaTefNikPit19}, the authors introduced \emph{paraboloid neurons} and illustrate their efficiency in comparison with conventional affine linear neural networks in a number of applications. In \cite{FanXioWan20}, the authors proposed  circular neurons  and deep quadratic networks.
The approaches of \cite{TsaTefNikPit19,FanXioWan20} are conceptually similar to the idea of this paper, where we replace the affine linear functions $\set{p_j}$ by quadratic polynomials, leading to \emph{quadratic neural network functions} of the form
\begin{equation}\label{eq:quadratic_approximation}
G(\vx) := \sum_{j=1}^{N} \alpha_j\sigma\left(\vx^T A_j \vx + \vwm_j^T \vx +\theta_j\right) \text{ with } \alpha_j, \theta_j \in \R, \vwm_j \in \R^{n} \text{ and } A_j \in \R^{n \times n}.
\end{equation}
In comparison, neural networks considered here are shallow (meaning that they have only a few numbers of layers) the networks from \cite{FanXioWan20} can, theoretically, have an infinite number of layers. The paraboloid neurons from \cite{TsaTefNikPit19} are a subset of the quadratic neurons.

Clearly, the functions from \autoref{eq:quadratic_approximation} represent a more general class of function then shallow affine linear neural networks, and therefore it might be expected that the number of nodes $N$ for an approximation of a function $g$ might be lower than for an affine linear neural network as in \autoref{eq:classical_approximation}, which is indeed true
as we show numerically in \autoref{sec:numerics}. In particular we show numerically that a shallow quadratic neural network can even be as efficient as a deep affine linear neural network.
We essentially base our convergence (rates) analysis of approximation properties of quadratic neural networks on the fundamental results of \cite{Mha93,DenHan09,ShaCloCoi18}. In \cite{ShaCloCoi18} they concentrate on analyzing \emph{4-layer} affine linear neural networks (which is already considered deep): For comparison purposes, in our numerical examples, we therefore concentrate mainly on \emph{3-layer} (these are actually termed shallow) quadratic neural networks.
\footnote{
In this paper we make a count of numbers of layers as in \cite{ShaCloCoi18}: Analogously we refer to an affine linear L-layer network when it consists of input and output layers and $L-2$ hidden (inner) layers.}

Particular achievements of our paper are as follows:
\begin{itemize}
\item We highlight that quadratic neural networks can be implemented relative easily in \emph{TensorFlow} \cite{AbaAgaBamBreChe16_report} and \emph{Keras} \cite{Cho15} by \emph{customized} layers.
\item Compared with \cite{ShaCloCoi18} the number of both layers and neurons in our case is lower because of the quadratic neurons: they used a 4-layer network with a total of at least $(8d + 2)N$ ordinary linear neurons in it for the same approximation level.
\item Furthermore, the original version of \autoref{thm:ConvR} from \cite{ShaCloCoi18} has been applied to prove convergence for 4-layer networks. Their network deals with a manifold setting and the first layer is responsible to determine \emph{compact atlas maps}; see \autoref{fig:CompaNets}, where the left image corresponds to \cite[Figure 3]{ShaCloCoi18}. However, the \emph{compact atlas} is essential in their analysis, which is related to the fact that affine linear neurons of the form $\vx \to \varphi(\vx):= C_d \sigma(\omega^T \vx + \theta)$ \emph{cannot} satisfy \autoref{it4} in \autoref{def:wavelet}, which is $\int_{\R^n}\varphi(\vx)d\vx=1$, and in turn the results from \cite{DenHan09} cannot be applied in free space $\R^n$, but of course it applies, when it is constrained to a compact set, which is the case for some quadratic functions.
\end{itemize}
The paper presents a proof of concept and thus we restrict attention only to quadratic neural networks although generalizations to higher order neural networks (such as cubic) is quite straightforward.

\section{Generalized universal approximation theorem}
\label{sec:guat}
In this section we review the \emph{universal approximation theorem} as formulated by \cite{Cyb89} and prove a generalization.
To this end we also need to introduce some elementary definitions and notation:
\begin{notation}[Vectors]
For two integer numbers $m,n \in \N$ we always assume that $m \geq n$.
Line vectors in $\R^m$ and $\R^n$ are denoted by
\begin{equation*}
\vwm = (w^{(1)},w^{(2)},\cdots,w^{(m)})^T \text{ and }
\vx = (x_1,x_2,\cdots,x_n)^T, \text{ respectively.}
\end{equation*}
The same notation will apply to functions: $\vfm, \vec{f}$ are $m$, $n$-dimensional vector valued functions, respectively.
\end{notation}
\begin{notation}[$\mathcal{L}^1$ space]
Define the norm following from Equation (1.10) in \cite{BarCohDahDev08}
\begin{equation*}
\|f\|_{\mathcal{L}^1} := \inf\{\sum_{g\in D}|c_g| | f= \sum_{g\in D}c_g g\},
\end{equation*}
where $c_g$ are the coefficients of the wavelet expansion and $D$ is the set of wavelet functions. Notice that the notation $\mathcal{L}^1$ does not refer to the common $L^1$-function space and depends on the choice of the wavelet system. For more properties and details on this space see \cite[Remark 3.11]{ShaCloCoi18}.
\end{notation}

\begin{definition}[Discriminatory function] \label{de:disc_function}
Let $\I{n} =[0,1]^n$ denote the closed $n$-dimensional unit-cube.
		A function $\sigma : \R \to \R$ is called \emph{discriminatory} if every measure $\mu$ on $\I{n}$, which satisfies
		\begin{equation*}
			\int_{\I{n}} \sigma (\vw^T \vx + \theta)\,d\mu(\vx)=0 \quad \text{ for all } \vw \in \R^n \text{ and } \theta \in \R
		\end{equation*}
	    implies that $\mu \equiv 0$.
\end{definition}
Note that every non-polynomial function is \emph{discriminatory} (this follows from the results in \cite{LesLinPinScho93}).

\begin{example} \label{ex:sigmoid}
	The sigmoid function, defined by $\sigma(t) = \frac{1}{1+\e^{-t}}$ for all $t \in \R$,
    is discriminatory for the Lebesgue-measure.
\end{example}
With these basic concepts we are able to recall Cybenko's universal approximation result.
\begin{theorem}[\cite{Cyb89}]\label{le:LinearApproximation}
 Let $\sigma:\R \to \R^+$ be a continuous discriminatory function. Then, for
 every function $g \in C(\I{n})$ and every $\epsilon>0$, there exists a function
\begin{equation}\label{eq:Linear}
G_\epsilon(\vx) = \sum_{j=1}^{N}\alpha_j \sigma(\vw_j^T\vx + \theta_j) \qquad \text{ with } N \in \N, \alpha_j, \theta_j \in \R, \vw_j \in \R^n,
\end{equation}
satisfying
\begin{equation*}
|G_\epsilon(\vx)-g(\vx)|<\epsilon \text{ for all } \vx\in \I{n}.
\end{equation*}
\end{theorem}
In the following we formulate and prove a generalization of Cybenko's result, which requires again some elementary definitions:

\begin{definition}[$m$-dimensional universal approximation functions]
\label{de:guat}
Let $f^{(1)}$, $f^{(2)}$, $\cdots$, $f^{(m)} \in C(\I{n})$, and
denote $\vfm^T := (f^{(1)}, f^{(2)}, ..., f^{(m)})$. Then we call
    \begin{equation}\label{eq:GeneralFunctions}
    	\mathcal{D}: = \mathcal{D}({\bf f}):=\set{ \vx \to \vwm^T \vfm(\vx) +\theta:
        \vwm \in \R^{m}, \theta \in \R}
    \end{equation}
the set of \emph{decision functions} associated to ${\bf f}^T$.
\end{definition}

\begin{theorem}[Generalized universal approximation theorem]\label{th:general}
Let $\sigma:\R \to \R$ be a continuous discriminatory function and assume that $\vfm : \I{n} \to \R^{m}$ is injective (this in particular means that $n \leq m$) and continuous.

Then for every $g\in C(\I{n})$ and every $\epsilon>0$ there exists some function
\begin{equation}\label{eq:general}
G_\epsilon^\vfm(\vx) := \sum_{j=1}^{N} \alpha_j\sigma\left(\vwm_j^T\vfm(\vx) +\theta_j\right) \text{ with } \alpha_j, \theta_j \in \R \text{ and } \vwm_j \in \R^{m}
\end{equation}
satisfying
\begin{equation*}
  \abs{G_\epsilon^\vfm(\vx) - g(\vx)} < \epsilon \text{ for all }\vx\in \I{n}.
\end{equation*}
\end{theorem}

\begin{proof}
We begin the proof by noting that since $\vx \to \vfm(\vx)$ is injective (The injectivity of the continuous function $\vfm$ follows from invariance of domain, see e.g. \cite[Theorem 4.3]{Dei85b}.), the inverse function on the range of $\vfm$ is well-defined, and we write $\vfm^{-1}:\vfm(\I{n}) \subseteq \R^m \to \I{n} \subseteq \R^n$.


The proof that $\vfm^{-1}$ is continuous relies on the fact that the domain $[0, 1]^n$ of $\vfm$ is compact, see for instance \cite[Chapter XI, Theorem 2.1]{Dug78}. Then applying the Tietze–Urysohn–Brouwer extension theorem (see \cite{Kel55}) to the continuous function $g \circ \vfm^{-1} : \vfm(\I{n}) \to \R$, this can be extended continuously to $\R^m$. This extension will be denoted by $g^*:\R^m \to \R$.

We apply \autoref{le:LinearApproximation} to conclude that there exist $\alpha_j, \theta_j \in \R$ and $ \vwm_j \in \R^{m}$, $j=1,\ldots,N$ such that
\begin{equation*}
G^*(\vzm) := \sum_{j=1}^{N} \alpha_j \sigma(\vwm_j^T \vzm +\theta_j) \text{ for all } \vzm \in \R^{m}, \theta_j \in \R,
\end{equation*}
which satisfies
\begin{equation}\label{eq:1}
\abs{G^*(\vzm)-g^*(\vzm)} <\epsilon \text{ for all } \vzm \in \R^m.
\end{equation}
Then, because $\vfm$ maps into $\R^m$ we conclude, in particular, that
\begin{equation*}
\begin{aligned}
G^*(\vfm(\vx)) &=\sum_{j=1}^{N} \alpha_j \sigma(\vwm_j^T \vfm(\vx)+\theta_j) \text{ and }
  \abs{G^*(\vfm(\vx)) - g(\vx)}
   =\abs{G^*(\vfm(\vx)) - g^*(\vfm(\vx))} <\epsilon.
\end{aligned}
\end{equation*}
Therefore $G_\epsilon^\vfm(\cdot):= G^*(\vfm(\cdot))$ satisfy the claimed assertions.
\end{proof}

\section{Universal approximation theorem with quadratic functions}
\label{sec:quad}
In the following we introduce several classes of universal approximation functions
as defined in \autoref{de:guat}.

First, we observe that the definition of decision functions from \autoref{de:guat} generalizes the affine linear decision functions from \cite{Cyb89} (see \autoref{le:LinearApproximation}):
\begin{example}[Affine linear decision functions] \label{ex:linear} Let $n=m$ and $f^{(i)}(\vx)=x_i$ for all $i=1,\ldots,n$. Then
the set of decision functions is given by
\begin{equation*}
   \mathcal{D}_l: = \set{ \vx \in \mathcal{I}_{n} \to \vw^T \vx + \theta : \vw  \in \R^n, \theta \in \R}.
\end{equation*}
Note, that in this case our notation gives $\vwm=\vw$ and $\vxm=\vx$.
\end{example}

In the following we consider different kinds of quadratic functions:
\begin{definition}[Quadratic decision functions] \label{de:quadratic}
Let $m=n+1$ and let $$A = U \text{diag}(\sigma_1,\ldots,\sigma_n) V^T \in \R^{n \times n}$$
the singular value decomposition of $A$. The functions
\begin{equation} \label{eq:quadratic_f}
 f^{(i)}(\vx) =x_i \text{ for }  i=1,\ldots,n \quad \text{ and } \quad
 f^{(n+1)}(\vx) = \vec{x}^T A \vec{x}
\end{equation}
define the \emph{quadratic decision functions} associated to $A$. The set of such
is denoted by
\begin{equation}\label{eq:GF_elliptic}
   \mathcal{D}: = \mathcal{D}(A) := \set{ \vx \to \vwm^T \vfm(\vx) +\theta:
        \vwm \in \R^{n+1}, \theta \in \R}.
\end{equation}
Then if
\begin{itemize}
\item for all $i$, $\sigma_i \geq 0$ or for all $i$, $\sigma_i \leq 0$, then $\mathcal{D}$ is called the set of {\bf elliptic} decision functions. In particular, if for all $i$, $\sigma_i = 1$ and $U=V=I$, the unitary matrix, then ${\mathcal D}$ is called the set of {\bf circular} decision functions.
\item If all but one $\sigma_i$ have the same sign, and are all not equal to $0$, then
      ${\mathcal D}$ is called the set of {\bf hyperbolic} decision functions, and
\item if all $\sigma_i \neq 0$ and more than two $\sigma_i$ have positive and negative signs, respectively, then
      ${\mathcal D}$ is called the set of {\bf ultrahyperbolic} decision functions.
\item If exactly one $\sigma_i=0$, and all others have the same sign, then ${\mathcal D}$ is called the set of {\bf parabolic} decision functions.
\end{itemize}
\end{definition}
\begin{remark}
\begin{itemize}
\item Let $A = 0$ be the null-matrix, then the quadratic decision functions  associated to  $A$ are the affine linear decision functions.
\item For every matrix $A \in \R^{n \times n}$ we have
      \begin{equation} \label{eq:linear_quadratic}
      \mathcal{D}_l \subseteq \mathcal{D}(A).
      \end{equation}
\item
Consider a quadratic decision function with
\begin{equation}\label{def:ai}
A=\text{diag}(a_1, a_2,...,a_n ) \quad \text{where } a_i \in \R, a_i \neq 0.
\end{equation}
Note that $m=n+1$.
If $w^{(m)} \neq 0$ we define $\zeta_i := -\frac{w^{(i)}}{2a_i^2 w^{(m)}}$, for $i=1,\ldots,n$,
$\zeta = (\zeta_1,\ldots,\zeta_n)^T$ and $\nu := w^{(m)}\zeta^TA\zeta -\theta$. Consequently, the decision function can be written as
\begin{equation}\label{eq:ex1_nontrivial}
\begin{aligned}
\vwm^T \vfm (\vx) + \theta &= w^{(m)}\vx^T A \vx + \sum_{i=1}^n w^{(i)} x_i + \theta\\
&= w^{(m)} \left( \vx^T A \vx - 2 \sum_{i=1}^n \zeta_i a_i^2 x_i + \zeta^T A \zeta \right) - w^{(m)}\zeta^T A \zeta +\theta\\
&= w^{(m)} \norm{\vx-\zeta}_A^2 - \nu.
\end{aligned}
\end{equation}
\end{itemize}
\end{remark}

Since the set of affine linear decision functions is always a subset of the quadratic decision functions
(see \autoref{eq:linear_quadratic}) the following result follows from an application of \autoref{th:general} taking into account that the function ${\bf f}$ defined in \autoref{eq:quadratic_f} is injective.
\begin{corollary}[Universal approximation of quadratic decision functions] \label{th:qdf} Let $m=n+1$, $A \in \R^{n\times n}$ and let ${\bf f}$ be as defined in \autoref{eq:quadratic_f}. Suppose that the discriminatory function $\sigma: \R \to \R_+$ is Lipschitz continuous with Lipschitz constant $\lambda$.
Then for every $g\in C(\I{n})$ and every $\epsilon>0$ there exists some $N \in \N$ and some function
\begin{equation}\label{eq:ge}
\vx \in \I{n} \to G_\epsilon^{\vfm}(\vx) := \sum_{j=1}^{N} \alpha_j\sigma\left(\vwm_j^T \vfm (\vx) + \theta_j\right) \text{ with } \alpha_j \in \R,\vwm_{j} \in \mathbb{R}^{m} \text { and } \theta_{j} \in \mathbb{R}
\end{equation}
satisfying
\begin{equation*}
  \abs{G_\epsilon^{\vfm}(\vx) - g(\vx)} < \epsilon \text{ for all } \vx \in \I{n}.
\end{equation*}
\end{corollary}
Note that the assumption that $\sigma$ is Lipschitz continuous is needed in the proof of \autoref{th:qdf}.

As it is presented here, the universal approximation \autoref{le:LinearApproximation} and
\autoref{th:qdf} provide the existence of an approximating sequence for increasing $N$. The proof is not quantitative and only applicable for functions $g\in C(\I{n})$, that is for functions defined on the $n$-dimensional unit cube. The following section provides convergence rates results for the best approximation function with $N$ coefficients. On a technical level, it allows for approximating functions $g \in L^1(\R^n)$, that is in free space.

\section{Convergence rates for universal approximation of circular decision functions}
\label{sec:rates}
In the following we prove convergence rates of circular decision functions of the form $G_\epsilon^{\bf f}$ in the $\mathcal{L}^1$-norm. We recall that by construction, circular decision functions form a superset of the affine linear decision functions, and this generalization allows for more efficient approximations.

We follow the proof of convergence rates results from \cite{ShaCloCoi18} for affine linear decision functions and extend it to \emph{circular decision functions} in the following way:
\begin{enumerate}
\item We construct a wavelet frame from the set of circular decision functions $\mathcal{D}(I)$;
\item We apply general convergence rates for wavelet expansions to prove convergence rates of the
      {\bf best approximation} with respect to the circular frame of an arbitrary function $g \in \mathcal{L}^1(\R^n)$.
\end{enumerate}
For the sake of simplicity of presentation we avoid a presentation of general elliptic decision functions.

%
%

\begin{definition}[Circular wavelet frame] \label{de:cf}
	Let $r > 0$ and $\sigma$ is a discriminatory function as defined in \autoref{de:disc_function} such that $\int_{\R^n}\sigma(r^2-\|\vx\|^2)d\vx < \infty$. Then let
	\begin{equation}\label{eq:varphi}
		\vx \in \R^n \to \varphi(\vx):= C_d \sigma(r^2 - x_1^2 - x_2^2 - \dots-x_n^2),
	\end{equation}
	where $C_d$ is a normalizing constant such that $\int_{\R^n}\varphi(\vx)d\vx=1$.\\	
	Then we define for all $\vx,\vy \in \R^n$ and $k \in \Z$
	\begin{equation}\label{eq:S}
		\begin{aligned}
			S_k(\vx,\vy) := 2^k \varphi(2^\frac{k}{n}(\vx-\vy)) \text{ and }
			\psi_{k,\vy}(\vx) := 2^{-\frac{k}{2}} (S_k(\vx,\vy)- S_{k-1}(\vx,\vy)).
		\end{aligned}
	\end{equation}
\end{definition}

\begin{remark} We abbreviate $\psi:=\psi_{0,0}$. With this notation we see that
	\begin{align*}
		\psi_{k,\vy}(\vx) &=  2^{-\frac{k}{2}}(S_k(\vx,\vy)- S_{k-1}(\vx,\vy)) \\
		&= 2^{-\frac{k}{2}}\left(2^k\varphi(2^{\frac{k}{n}}(\vx-\vy)) - 2^{k-1}\varphi(2^{\frac{k-1}{n}}(\vx-\vy))\right)\\
		&= 2^{\frac{k}{2}} \left( \varphi \left(2^{\frac{k}{n}}\vx-2^{\frac{k}{n}}\vy\right) - 2^{-1} \varphi\left(2^{-\frac{1}{n}} \left(2^{\frac{k}{n}}\vx - 2^{\frac{k}{n}}\vy\right)\right)\right)\\
		&= 2^{\frac{k}{2}} \psi(2^{\frac{k}{n}}(\vx-\vy)) \qquad \text{ for all } k \in \Z, \vy \in \R^n.
	\end{align*}
\end{remark}
For proving that ${\mathcal F}$ satisfies the frame properties (see for instance \cite{Chr16})
and approximation properties of the best approximation with respect to the frames expansion we apply some general results from the literature, which are reviewed in the \autoref{sec:appendix}.

\subsection{Convergence rates of nets of circular decision function}
We show that the circular wavelet frame is an \emph{Approximation of the identity} (AtI (see \autoref{def:wavelet})). For this purpose we use the following basic inequality.
\begin{lemma}\label{lem:hessian}
Let $h: \R^n \to \R$ be a twice differentiable function, which can be expressed in the following way:
\begin{equation*}
   h(\vx)=h_s(\norm{\vx}^2) \text{ for all } \vx \in \R^n.
\end{equation*}
Then the spectral norm of the Hessian of $h$ can be estimated as follows:\footnote{In the following $\nabla$ and $\nabla^2$ (without subscripts) always denote derivatives with respect to an $n$-dimensional variable such as $\vx$. ${}'$ and ${}''$ denotes derivatives of a one-dimensional function.}
\begin{equation} \label{eq:hessian}
  \norm{\nabla^2 h(\vx)} \leq \max \set{\abs{4\norm{\vx}^2 h_s''(\norm{\vx}^2)+ 2h_s'(\norm{\vx}^2)},\abs{2h_s'(\norm{\vx}^2)}}.
\end{equation}
\end{lemma}
\begin{proof}
Since $\nabla^2 h(\vx)$ is a symmetric matrix, its operator norm is equal to its spectral radius, namely the largest absolute value of an eigenvalue. By routine calculation we can see that
\[
\nabla_{x_i x_j}h(\vx)=4x_ix_jh_s''(\norm{\vx}^2) + 2\delta_{ij}h_s'(\norm{\vx}^2).
\]
Let $C =(x_ix_j)$ and $I$ the identity matrix, then $\lambda$ is an eigenvalue with eigenvector $\vz$ of $\nabla^2 h(\vx)$ if and only if
\begin{equation*}
	4 h_s''(\norm{\vec{x}}^2) C \vz = (-2 h_s'(\norm{\vx}^2) + \lambda)\vz.
\end{equation*}
Or in other words $\frac{- 2 h_s'(\norm{\vx}^2) + \lambda}{4 h_s''(\norm{x}^2)}$ is an eigenvalue of $C$. Moreover, $C = \vx \vx^T$ is a rank one matrix and thus the spectral values are $0$ with multiplicity $(n-1)$ and $\norm{\vx}^2$. This in turn shows that the eigenvalues of the Hessian
are $+2h_s'(\norm{\vx}^2)$ (with multiplicity $n-1$) and $4\norm{\vx}^2 h_s''(\norm{\vx}^2)+2h_s'(\norm{\vx}^2)$, which proves \autoref{eq:hessian}.

\end{proof}

In the following lemma, we will prove that the kernels $(S_k)_{k \in \Z}$ are an AtI (Approximation to the identity \cite{DenHan09}). This is a streamlined assumption from Definition 3.4 in the book \cite{DenHan09}.

\begin{lemma}\label{le:SConditions}
Suppose that the activation function $\sigma:\R \to \R_+$ is monotonically increasing
and satisfies for the $i$-th derivative ($i=0,1,2$)
\begin{equation} \label{eq:sigma}
   \abs{\sigma^i (r^2-t^2)} \leq C_\sigma (1+\abs{t}^n)^{-1-\frac{2i+1}{n}} \text{ for all } t \in \R,
\end{equation}
where $r$ is the same as in \autoref{de:cf}. Then the kernels $(S_k)_{k \in \Z}$ as defined in \autoref{eq:S} form an AtI as defined in \autoref{def:wavelet} that also satisfy \autoref{eq:DoubleLipschitzCondition}.
\end{lemma}

\begin{proof}
We verify the three conditions from \autoref{def:wavelet} as well as \autoref{eq:DoubleLipschitzCondition}.
First of all, we note that
\begin{equation} \label{eq:sigma1}
   \abs{\sigma^i (r^2 - \norm{\vx}^2)} \leq C_\sigma (1+\norm{\vx}^n)^{-1-\frac{2i+1}{n}} \text{ for all } \vx \in \R^n.
\end{equation}
\begin{itemize}
  \item \textbf{Verification of \autoref{it1} in \autoref{def:wavelet}:}
   \autoref{eq:varphi} and \autoref{eq:sigma} imply that
  \begin{equation}\label{eq:VarphiBound}
   0 \leq \varphi(\vx-\vy) = C_d \sigma(r^2-\norm{\vx-\vy}^2) \leq C_\sigma C_d (1+\norm{\vx-\vy}^n)^{-1-\frac{1}{n}} \text{ for all } \vx,\vy \in \R^n.
  \end{equation}
  Therefore
  \begin{equation*} \begin{aligned}
    S_k(\vx,\vy)&= 2^k \varphi(2^\frac{k}{n}(\vx-\vy))
    \leq C_\sigma C_d 2^{k}(1+2^k \norm{\vx-\vy}^n)^{-1-\frac{1}{n}} \\
    &= C_\sigma C_d 2^{-\frac{k}{n}}(2^{-k}+ \norm{\vx-\vy}^n)^{-1-\frac{1}{n}}.
  \end{aligned} \end{equation*}
  Thus \autoref{it1} in \autoref{def:wavelet} holds with $\epsilon=1/n$ and $C_\rho=1$ and $C=C_d C_\sigma$.
  \item \textbf{Verification of \autoref{it2} in \autoref{def:wavelet} with $C_\rho=1$ and $C_A=2^{-n}$:}
      Because $\sigma$ is monotonically increasing it follows from \autoref{eq:varphi} and the fact that $S_{0}(\vec{x}, \vec{y})=\varphi(\vec{x}-\vec{y})$  (see \autoref{eq:S}) and \autoref{de:cf} that
      $$F_\vy(\vx):= \norm{\nabla_\vx (S_0(\vx,\vy))} = 2 C_d \norm{\vx-\vy}\sigma'(r^2-\norm{\vx-\vy}^2) \text{ for all } \vy \in \R^n.$$
      Then \autoref{eq:sigma1} implies that
      \begin{equation*}
      \begin{aligned}
      F_\vy(\vx) &\leq 2C_d C_\sigma (1+\norm{\vx-\vy}^n)^{-1-\frac{3}{n}} \norm{\vx-\vy} \leq
      2C_d C_\sigma (1+\norm{\vx-\vy}^n)^{-1-\frac{3}{n}} (1+\norm{\vx-\vy}^n)^{\frac{1}{n}}\\
      & = 2C_d C_\sigma (1+\norm{\vx-\vy}^n)^{-1-\frac{2}{n}}.
      \end{aligned}
  \end{equation*}
      From the definition of $S_k(\vx,\vy)$, it follows
      \begin{equation} \label{eq:zw}
      \begin{aligned}
      \norm{\nabla_\vx (S_k(\vx,\vy))} &= \norm{\nabla_\vx (2^k\varphi(2^{\frac{k}{n}}(\vx-\vy)))} =
       2^k \norm{\nabla_\vx S_0 (2^{\frac{k}{n}}\vx,2^{\frac{k}{n}}\vy)} = 2^{k+\frac{k}{n}}F_{2^{\frac{k}{n}}\vy} (2^{\frac{k}{n}}\vx)\\
      &\leq 2^{-\frac{k}{n}} C_d C_\sigma (2^{-k}+\norm{\vx-\vy}^n)^{-1-\frac{2}{n}}.
      \end{aligned}
      \end{equation}
  From the mean value theorem it therefore follows from \autoref{eq:zw} and \autoref{eq:kk} that
      \begin{equation} \label{eq:help}
      \begin{aligned}
      \frac{\abs{S_k(\vx,\vy) - S_k(\vx',\vy)}}{\norm{\vx-\vx'}} &\leq
      \max_{\set{\vz = t \vx' + (1-t)\vx:t \in [0,1]}} \norm{\nabla_\vx(S_k(\vz,\vy))}\\
      &\leq 2^{-\frac{k}{n}} C_d C_\sigma \max_{\set{\vz = \vx + t (\vx' -\vx):t \in [0,1]}} (2^{-k}+\norm{\vz-\vy}^n)^{-1-\frac{2}{n}} \\
      &= 2^{-\frac{k}{n}} C_d C_\sigma \left(2^{-k} + \min_{\set{\vz = \vx + t (\vx' -\vx):t \in [0,1]}} \norm{\vz-\vy}^n \right)^{-1-\frac{2}{n}}.
      \end{aligned}
      \end{equation}
      Then application of \autoref{eq:kk} and noting that $\frac{1-2^{-n}}{2^{-n}} \geq 1$
      gives
      \begin{equation*}
      	\begin{aligned}
      \frac{\abs{S_k(\vx,\vy) - S_k(\vx',\vy)}}{\norm{\vx-\vx'}} &\leq
      2^{-\frac{k}{n}} C_d C_\sigma \left((1-2^{-n}) 2^{-k}+2^{-n}\norm{\vx-\vy}^n  \right)^{-1-\frac{2}{n}}\\ &\leq
      2^{-\frac{k}{n}} \left(2^{-n}\right)^{-1-\frac{2}{n}}
      C_d C_\sigma \left(\frac{(1-2^{-n}) }{2^{-n}}2^{-k} + \norm{\vx-\vy}^n  \right)^{-1-\frac{2}{n}}\\&\leq
      2^{-\frac{k}{n}} 2^{n+2}
      C_d C_\sigma \left(2^{-k} + \norm{\vx-\vy}^n  \right)^{-1-\frac{2}{n}}.
      	\end{aligned}
      \end{equation*}
      Therefore \autoref{it2} is satisfied with $C_\rho=1$, $\zeta=1/n$, $\epsilon=1/n$ and
      $C=2^{n+2} C_d C_\sigma $.

  \item \textbf{Verification of \autoref{it4} in \autoref{def:wavelet}:}
      From the definition of $S_k$ (see \autoref{eq:S}) it follows that for every $k\in \Z$ and $\vy \in \R^n$
  \begin{equation*}
  1 = \int_{\R^n}S_k(\vx,\vy)d\vx = \int_{\R^n} 2^k \varphi(2^\frac{k}{n}(\vx-\vy)) d\vx.
  \end{equation*}
\item \textbf{Verification of the double Lipschitz condition \autoref{eq:DoubleLipschitzCondition} in \autoref{def:wavelet}:}
By using the integral version of the mean value theorem, we have
\begin{align*}
  &S_k(\vx,\vy) - S_k(\vx',\vy) - S_k(\vx,\vy') + S_k(\vx',\vy') \\
  &= S_k(\vx,\vy) - S_k(\vx',\vy) - (S_k(\vx,\vy') - S_k(\vx',\vy'))\\
  &= \int_{0}^{1} \langle \nabla_\vy S_k(\vx,\vy' + t(\vy-\vy')), \vy-\vy' \rangle dt - \int_{0}^{1} \langle \nabla_\vy S_k(\vx',\vy' + t(\vy-\vy')), \vy-\vy' \rangle dt\\
  &= \int_{0}^{1}\int_{0}^{1} \langle \nabla_{\vx, \vy} S_k(\vx' + s(\vx-\vx'),\vy' + t(\vy-\vy'))(\vx-\vx'), \vy-\vy' \rangle dt ds.
\end{align*}
Following this identity, we get
\begin{equation} \label{eq:lip2}
  \begin{aligned}
   \frac{\abs{S_k(\vx,\vy) - S_k(\vx',\vy) + S_k(\vx,\vy') - S_k(\vx',\vy')}}
         {\norm{\vx-\vx'} \norm{\vy-\vy'}}
       \leq &  \max_{\interval}
       \norm{\nabla_\vy \left(\frac{S_k(\vx,\vz) - S_k(\vx',\vz)}{\norm{\vx-\vx'}}\right)}\\
    \leq & \max_{\quader} \norm{\nabla^2_{\vx\vy}S_k(\vz',\vz)},
  \end{aligned}
\end{equation}
where
\begin{equation*}
  \begin{aligned}
       \interval &:= \set{\vz=t\vy+(1-t)\vy':t \in [0,1]}, \\
       \quader &:= \set{(\vz=t_{\vy}\vy+(1-t_{\vy})\vy',\vz'=t_{\vx}\vx+(1-t_{\vx})\vx'):
                      t_{\vy} \in [0,1],t_{\vx} \in [0,1]},
  \end{aligned}
\end{equation*}
and $\norm{\nabla^2_{\vx\vy}S_k(\vz',\vz)}$ denotes again the spectral norm of $\nabla^2_{\vx\vy}S_k(\vz',\vz)$.

Now, we estimate the right hand side of \autoref{eq:lip2}: From the definition of $S_k$, \autoref{eq:S}, and the definition of $\varphi$, \autoref{eq:varphi}, it follows with the abbreviation
$\vec{\omega} = 2^{\frac{k}{n}}(\vz'-\vz)$:
\begin{equation*}
\begin{aligned}
\norm{\nabla^2_{\vx\vy}S_k(\vz',\vz)} &=
2^{k}\norm{\nabla^2_{\vx\vy}(\varphi \circ (2^{\frac{k}{n}}\cdot))(\vz'-\vz))}
= 2^{k+2\frac{k}{n}}\norm{\nabla^2 \varphi (\vec{\omega}))}.
\end{aligned}
\end{equation*}
Applications of \autoref{lem:hessian} with $\vx \to h(\vx)=\varphi(\vx)$ and $t \to h_s(t) = C_d \sigma(r^2-t)$ shows that (note that $h_s'(t)=-C_d \sigma'(r^2-t)$)
  \begin{equation}\label{eq:last}
   \begin{aligned}
   \norm{\nabla^2 \varphi (\vec{\omega})}
  \leq & C_d  \max \set{\abs{4\norm{\vec{\omega}}^2 \sigma''(r^2 - \norm{\vec{\omega}}^2) -  2\sigma'(r^2 - \norm{\vec{\omega}}^2)},\abs{2\sigma'(r^2 - \norm{\vec{\omega}}^2)}}\\
  \leq & 2^2 C_d  \max \set{2\norm{\vec{\omega}}^2 \abs{\sigma''(r^2 - \norm{\vec{\omega}}^2)},\abs{\sigma'(r^2 - \norm{\vec{\omega}}^2)}}.
      \end{aligned}
 \end{equation}
Thus from \autoref{eq:sigma1} it follows that
\begin{equation*}
	\begin{aligned}
		\norm{\nabla^2_{\vx\vy}S_k(\vz',\vz)}
		\leq & 2^2 2^{k+2\frac{k}{n}}C_d C_\sigma  \max \set{2\norm{\vec{\omega}}^2 (1+\norm{\vec{\omega}}^n)^{-1-\frac{5}{n}}, (1+\norm{\vec{\omega}}^n)^{-1-\frac{3}{n}} }\\
		\leq & 2^3 2^{k+2\frac{k}{n}}C_d C_\sigma  (1+\norm{\vec{\omega}}^n)^{-1-\frac{3}{n}}\\
		\leq & 2^{-\frac{k}{n}} 2^3 C_d C_\sigma  (2^{-k} + \norm{ \vz'-\vz}^n)^{-1-\frac{3}{n}} .
	\end{aligned}
\end{equation*}
In the next step we note that from \autoref{eq:kk1} it follows that
\begin{equation}\label{eq:min}
\norm{\vz'-\vz}^n \geq 3^{-n} \norm{\vx-\vy}^n - 3^{-n} 2^{1-k}.
\end{equation}

Thus we get because $\frac{1-3^{-n}2}{3^{-n}} \geq 1$
\begin{equation*}
    \begin{aligned}
    ~ & \frac{\abs{S_k(\vx,\vy) - S_k(\vx',\vy) - S_k(\vx,\vy')  + S_k(\vx',\vy')}}
         {\norm{\vx-\vx'} \norm{\vy-\vy'}} \\
    \leq & 2^{-\frac{k}{n}} 2^3 C_d C_\sigma \left((1-3^{-n}2) 2^{-k} + 3^{-n}\norm{\vx-\vy}^n \right)^{-1-\frac{3}{n}}\\
\leq & 2^{-\frac{k}{n}} 2^3 (3^{-n})^{-1-\frac{3}{n}} C_d C_\sigma
     \left(\frac{1-3^{-n}2}{3^{-n}} 2^{-k} + \norm{\vx-\vy}^n \right)^{-1-\frac{3}{n}}\\
\leq & 2^{-\frac{k}{n}} 2^3 3^{n+3} C_d C_\sigma
     \left(2^{-k} + \norm{\vx-\vy}^n \right)^{-1-\frac{3}{n}}.
    \end{aligned}
\end{equation*}
Therefore \autoref{it4} is satisfied with $C_\rho=1$,
$\tilde{C} = 2^3 3^{n+3} C_d C_\sigma$, $\zeta=1/n$, and $\epsilon=1/n$.
  \end{itemize}
\end{proof}

\begin{remark}
One typical activation function which satisfies \autoref{eq:sigma} is the sigmoid function. For different orders of its derivative, \autoref{fig:SigmoidInequality} shows this inequality in logarithmic coordinates. The orange curve corresponds to the left hand side of \autoref{eq:sigma}, and the blue curve corresponds to the right hand side of it.
\begin{figure}[H]%
    \protect
    \centering
    \subfloat[when $i=0$]{\includegraphics[width=7cm]{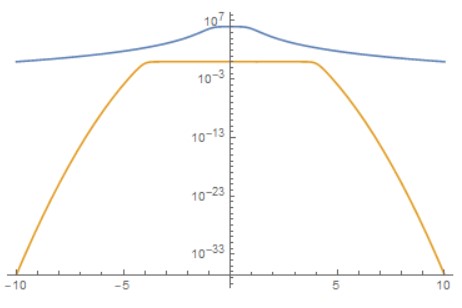} }%
    \qquad
    \subfloat[when $i=1$]{\includegraphics[width=7cm]{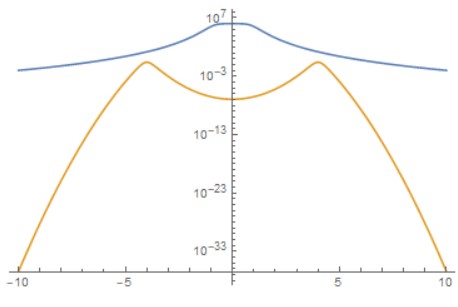} }%
    \qquad
    \subfloat[when $i=2$]{\includegraphics[width=7cm]{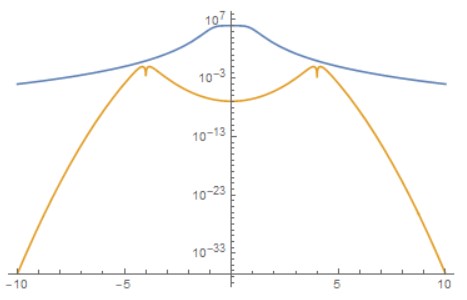} }%
    \caption{In \autoref{eq:sigma}, we take $\sigma(t)=\frac{1}{1+e^{-t}}$, $C_\sigma=10^6$, $n=5$, $r=4$. The x-axis is variable $t$, the y-axis is the logarithmic value of both sides of \autoref{eq:sigma}.}
    \label{fig:SigmoidInequality}%
\end{figure}
\end{remark}

By combining \autoref{le:WaveletApprox} and \autoref{le:SConditions}, we get the following theorem:
\begin{theorem}[$\mathcal{L}^1$-convergence] \label{thm:ConvR}
Let $\sigma$ be an activation function that satisfies the conditions in \autoref{le:SConditions}, and let $\psi_{k,b}$ be a frame constructed from $\sigma$ by \autoref{de:cf}. For any function $f\in \mathcal{L}^1(\R^n)$ and any positive integer $N$, there exists a function
$$f_N \in \text{span}_N(\mathcal{F}) \subseteq \mathcal{L}^1(\R^n) \text{ where } \mathcal{F}:=  \set{\vx \in \R^n \to \psi_{k,b}(\vx): (k,b) \in \Z , b \in 2^{-\frac{k}{n}}\Z }$$ and $\text{span}_N$ denotes linear combinations of at most $N$ terms in the set, such that
\begin{equation}
\label{eq:app_error}
\norm{f-f_N}_{L^2} \leq \norm{f}_{\mathcal{L}^1} (N+1)^{-1/2}.
\end{equation}

\end{theorem}
\begin{proof}
First, we note that the functions $\set{S_k: k \in \Z}$ are an AtI, which satisfies the double Lipschitz condition (see \autoref{def:wavelet}). Thus associated to \autoref{le:WaveletApprox}
${\mathcal F}$ is a wavelet frame.
Moreover, let $f_N$ be the wavelet approximation specified in \autoref{le:WaveletApprox}, then
it satisfies \autoref{eq:app_error}.
\end{proof}

\begin{remark}\label{rem:NumberOfNeurons}
An original version of \autoref{thm:ConvR} has been applied to prove convergence for four layer networks (note, this means two hidden, one input and one output layer). Their network deals with a manifold setting and the first layer is responsible to determine {\bf compact atlas maps}; see \autoref{fig:CompaNets}, where the left image corresponds to \cite[Figure 3]{ShaCloCoi18}. However, the {\bf compact atlas} is essential in their analysis, which is related to the fact that affine linear neurons of the form $\vx \to \varphi(\vx):= C_d \sigma(\omega^T \vx + \theta)$ {\bf cannot} satisfy \autoref{it4} in \autoref{def:wavelet}, which is $\int_{\R^n}\varphi(\vx)d\vx=1$, and in turn the results from \cite{DenHan09} cannot be applied in free space $\R^n$, but of course it applies, when it is constrained to a compact set.
\begin{figure}[H]%
    \protect
    \centering
    \subfloat[Shaham et. al.’s neural network structure]{\includegraphics[width=7cm]{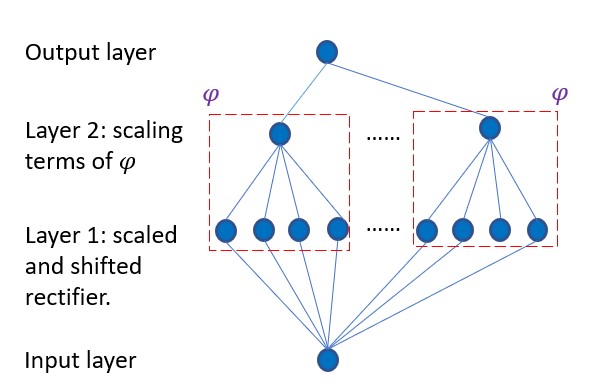} }%
    \qquad
    \subfloat[Our neural network structure]{\includegraphics[width=7cm]{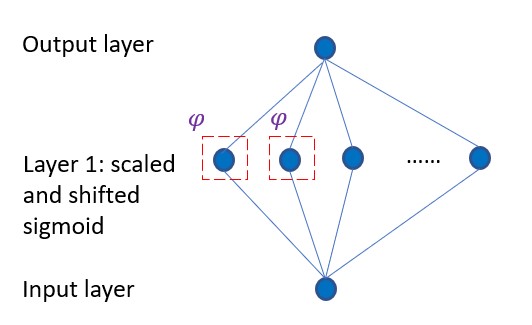} }%
    \caption{Comparison and improvement}
    \label{fig:CompaNets}%
\end{figure}
\end{remark}
In Lemma 4.5 in \cite{ShaCloCoi18} also an $L^\infty$ approximation result is proven,
which can be carried over to our setting as well.

\section{Numerical experiments} \label{sec:numerics}
In this section we study numerically the approximation of functions with linear combinations of quadratic decision functions, in particular circular and hyperbolic ones, as defined in \autoref{de:quadratic}.
We compare the numerical results with results produced by approximation with
affine linear decision functions.

Moreover, we compare deep affine linear neural networks with quadratic neural networks with shallow layer structure, that is a three-layer network (one hidden layer) (cf. right image of \autoref{fig:CompaNets}).

We also compare numerically the approximation properties of deep quadratic neural networks, which is not considered theoretically here.
We have chosen two simple test cases of one-dimensional functions (see \autoref{fig:gd}) to approximate, for which we analyze the approximation properties of different neural networks numerically.

Finally, our goal is to extend the basic proof-of-concept examples to some clusterization problems and discuss some real world application examples.

\subsection{Proof-of-concept example}

\subsubsection{Ground Truth Data}
We study the numerical approximation of two simple test-data, which are analytically given by
\begin{equation}
        f_1(x)=
        \begin{cases}
            0 &x \leq 0 \\
            x &0 < x \leq 3 \\
            3 &3 < x \leq 5 \\
            -0.4 x + 5 &5 < x \\
        \end{cases} \qquad
        f_2(x)=
        \begin{cases}
            0 &x \leq 0 \\
            x^2 &0 < x \leq 3 \\
            9 &3 < x \leq 5 \\
            9 \e^{-(x-5)} &5 < x \\
        \end{cases}
\end{equation}
\begin{figure}[H]%
    \protect
    \centering
    \subfloat[$f_1$]{\includegraphics[width=7cm]{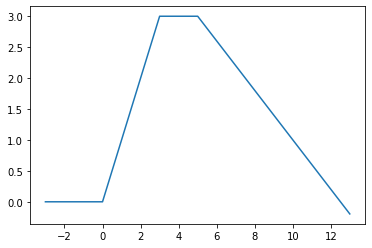} }%
    \qquad
    \subfloat[$f_2$]{\includegraphics[width=7cm]{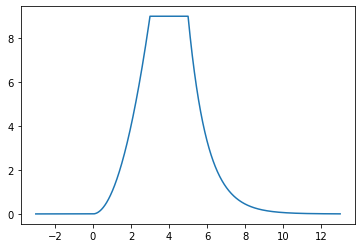} }%
    \caption{The two functions that are approximated via different types of neural networks in our numerical experiments below.
    }
    \label{fig:gd}%
\end{figure}

\subsubsection{Generation of training data and initialization}
The functions from above were evaluated in $1600$ uniformly distributed points in $[-3,13]$. 
The pairs $(x,y)$ of generated data were then randomly split into $1072$ training pairs and $528$ test pairs. In the rest of the paper $f_2$ is observed, but the analysis has also been conducted for $f_1$, where the results are similar.

\subsubsection{Implementation details}

The implementation is based on already implemented methods of {\bf \emph{TensorFlow}} \cite{AbaAgaBamBreChe16_report} and {\bf \emph{Keras}} \cite{Cho15}, with adaptations - where necessary - to resemble the structure of the decision functions from \eqref{eq:GF_elliptic}. The following pseudocode illustrates the general procedure and highlights custom implemented features.

\begin{algorithm}
\caption*{Structure of an elliptical layer}
\begin{algorithmic}

\State \textbf{Inputs:}
\State $\quad$ inputs $\leftarrow$ function values of ground truth data
\State $\quad$ input\_dimension $\leftarrow$ numer of neurons of input
\State $\quad$ output\_dimension $\leftarrow$ number of neurons of output

\State \textbf{Variables:}

\State $\quad$ \HiLi $\sigma_i$ $\leftarrow$ weights for the 'elliptical' part (dim of matrix: input dimension $\times$ number of neurons)
\State  $\quad$  $\vwm_{1..n}$ \ $ \leftarrow$ weights for the 'affine linear' part (dim of matrix: input\_dimension $\times$ number of neurons)
\State $\quad$  \HiLi $\vwm_{n+1}$ $\leftarrow$ weight that is multiplied with 'elliptical part' for each neuron (dim of matrix: 1 $\times$ number of neurons)
\State  $\quad$  $\theta$ $\leftarrow$ bias for each neuron to add to 'affine linear part' (dim of matrix: 1 $\times$ number of neurons)

\State \textbf{Initialize:}

\State  $\quad$  $\sigma_i$, $\vwm_{1..n}$, $\vwm_{n+1}$  $\leftarrow$ random gaussian distribution (mean=1.0, stddev=0, seed=None)
\State  $\quad$  $\theta$ $\leftarrow$ 0

 \Class{Elliptical\_layer(input\_dimension, output\_dimension)}

 \Function{call(inputs)}
  \State \HiLi outputs $\leftarrow$ $\vwm_{1..n}$ $\cdot$ inputs + $\vwm_{n+1}$ $\cdot$  $\sigma_i^2$ $\cdot$ inputs$^2$   + $\theta$
   \State     store outputs
    \State    return outputs
  \EndFunction

   \Function{backprop(labels):}
  \State update weights with Adam optimizer
    \EndFunction

\EndClass

\end{algorithmic}
\end{algorithm}

\begin{algorithm}
\caption*{Stacking together different layers and training (Example for a deep neural network)}
\begin{algorithmic}

\State \textbf{Initialize:}
\State $\quad$ $x_{train},y_{train}$ $\leftarrow$ training data
\State $\quad$ $test_{epochs}$ $\leftarrow$ number of complete pass throughs of the training data
\State $\quad$ layer\_in $\leftarrow$ Object of Elliptic class with 1 input, $X_1$ outputs
\State $\quad$ layer\_hidden\_1 $\leftarrow$ Object of Elliptic class with $X_1$ input, $Y_n$ outputs
\State $\quad \quad $ $\vdots$
\State $\quad$ layer\_hidden\_n $\leftarrow$ Object of Elliptic class with $X_n$ input, $Y_n$ outputs
\State $\quad$ layer\_out $\leftarrow$ Object of Elliptic class with $Y_n$ input, 1 outputs

  \State \HiLi network = List(layer\_in, layer\_hidden\_1, $\,\ldots\,$ ,layer\_hidden\_n, layer\_out)


 \State \textbf{Output:} $\mathrm{network}_{loss} = \textrm{network}.fit(x_{train},y_{train}, epochs=test_{epochs})$

\end{algorithmic}
\end{algorithm}

The implementation in our specific setting has been done for one, three and four hidden layers, but, as the pseudocode demonstrates, it can easily  be adapted to neural networks with any number of hidden layers. The Adam optimizer with a learning rate of $0.001$ has been selected as the optimization method of choice.
If we determine a bad initialization, it randomly chooses another one and starts again.

The following results shown in the next subsections have all been performed on a 2,4 GHz 8-Core Intel Core i9 processor with 32 GB RAM.

\subsubsection{Convergence rates of shallow elliptic networks}
We have varied the number of neurons $N$ to evaluate the convergence rates proven in  \autoref{thm:ConvR}, which shows the predicted convergence rates of elliptic neural networks (generalization of circular decision functions, which lead to more stable results). The results can be observed in \autoref{fig:rates}.

 \begin{figure}[H]%
    \protect
    \centering
    \includegraphics[width=7cm]{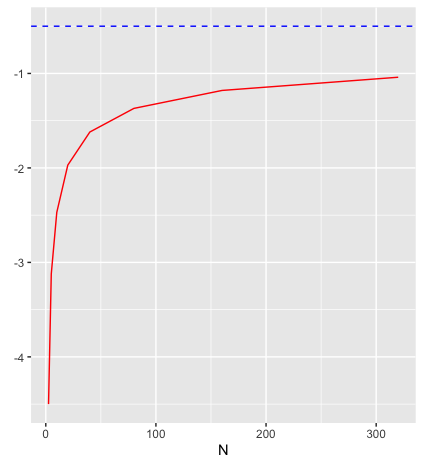} %
    \caption{The graph depicts the convergence rates of $\frac{\log(\frac{\norm{f-f_N}_{L^2}}{\norm{f}_{L^1}})}{\log(N+1)}$ in our numerical experiments, which are always below the upper bound given by the theoretical results.
    This means that actually the estimate \autoref{eq:app_error} seems to be too conservative.
}%
    \label{fig:rates}%
\end{figure}

In a next step, the convergence rates of elliptic neural networks should be compared with those of affine linear and hyperbolic ones. The following graph shows the development of the reconstruction each adding up 50 epochs.
One epoch describes one complete pass through the training data, as in the {\bf\emph{Keras}} library. During each epoch, the weights (including $\sigma_{i}$) are updated.

The images in \autoref{fig:development_epochs} show us, that when one considers training for 50 epochs, only approximation via elliptic (generalization of circular) decision functions performs reasonably well. When using more epochs, hyperbolic and affine linear catches up with elliptic and have a similar approximation error (see also \autoref{table:shallow}). \autoref{fig:error_function} provides us with the error function, where it is clearly observable that the elliptic layers converge faster than the affine linear and hyperbolic ones.

The MSE error corresponds with the ${L^2}$-norm from the theoretical section.

\begin{table}[H]
\begin{tabular}{|c|c|c|c|c|c|}
 \hline
 Type of network
 &
 Training epochs
 &
 Hidden layers
 &
 Units per hidden layer
 &
 Test MSE
 &
 Test MAE\\
 \hline \hline 
 \multirow{2}{4em}{Elliptic} & 250 & 1 & [5] & 0.0731&0.2592 \\
  &   140  & 1   &  [5] & 0.08795&0.25975 \\
 \hline
 \multirow{2}{6em}{Affine Linear} & 250 & 1 &  [5]&0.2008 &0.2102 \\
   &   140  & 1   &  [5] & 0.47645&0.30545 \\
 \hline
  \multirow{2}{4em}{Hyperbolic} & 250 & 1 &  [5]&0.0574 &0.2977 \\
    &   140  & 1   &  [5] & 0.2164&0.3281 \\
 \hline

\end{tabular}
\caption{Overview of the result of the numerical experiments for shallow networks.}
    \label{table:shallow}%
\end{table}

\begin{figure}[H]%
    \protect
    \centering
        \subfloat[Approximation of the function after $50$ epochs]{{\includegraphics[width=7cm]{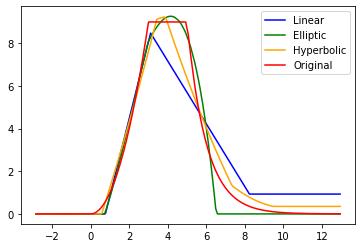} }}%
    \qquad
        \subfloat[Approximation of the function after $100$ epochs]{{\includegraphics[width=7cm]{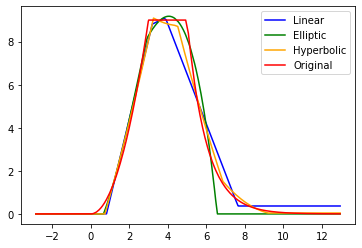} }}%
    \qquad
        \subfloat[Approximation of the function after $150$ epochs]{{\includegraphics[width=7cm]{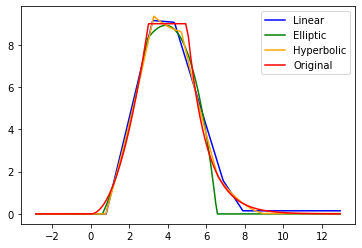} }}%
     \qquad
        \subfloat[Approximation of the function after $200$ epochs]{{\includegraphics[width=7cm]{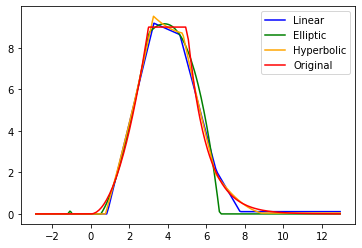} }}%
    \caption{The graphs each show the original function and the approximations obtained with shallow neural networks with each $5$ units per hidden layer, with gradual increase of $50$ epochs each.}%
    \label{fig:development_epochs}%
\end{figure}

\begin{figure}[H]%
    \protect
    \centering
    \subfloat{{\includegraphics[width=6cm]{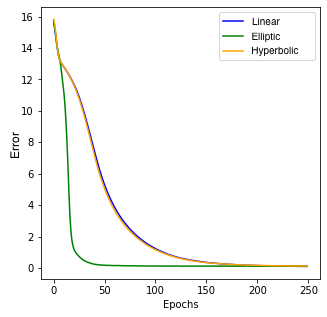} }}%
    \caption{Here we compare the error functions (MSE) of affine linear, elliptic and hyperbolic shallow neural networks. Elliptic layers clearly converge faster than affine linear or hyperbolic layers in this setting. }%
    \label{fig:error_function}%
\end{figure}

Please note, that the approximation with hyperbolic and affine linear layers do not satisfy all proposed conditions on $\sigma$ (see Lemma \ref{le:SConditions}).

\subsubsection{Deep networks}
In this section, we present the error functions for deep neural networks with multiple hidden layers, associated to the pseudocode presented before.

The following results (both in \autoref{fig:deep_network1} and \autoref{fig:deep_network2}) confirm our hypothesis, that the deep elliptic neural networks converge faster than the affine linear ones. Full results can be observed in \autoref{table:deep3}.
\begin{figure}[H]%
    \protect
    \captionsetup[subfigure]{justification=centering}
    \centering
        \subfloat[{Each hidden layer has $5$ neurons included, here we compare a 3-layer elliptic network with a 4-layer affine linear one.}][{Each hidden layer has $5$ neurons included,\\here we compare a 3-layer elliptic \\network with a 4-layer affine linear one.}]{\includegraphics[width=7cm]{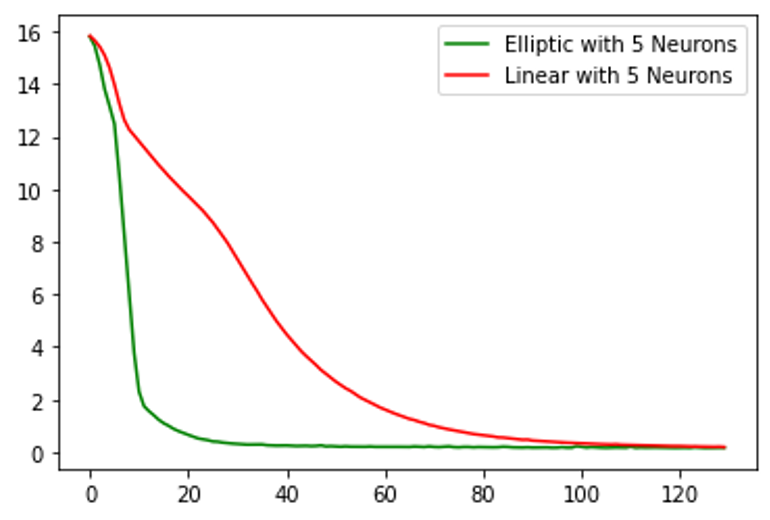} } %
    \qquad
        \subfloat[Each hidden layer has $30$ neurons included,\\here we compare a 3-layer elliptic \\network with a 4-layer affine linear one.][Each hidden layer has $30$ neurons included,\\here we compare a 3-layer elliptic \\network with a 4-layer affine linear one.]{{\includegraphics[width=7cm]{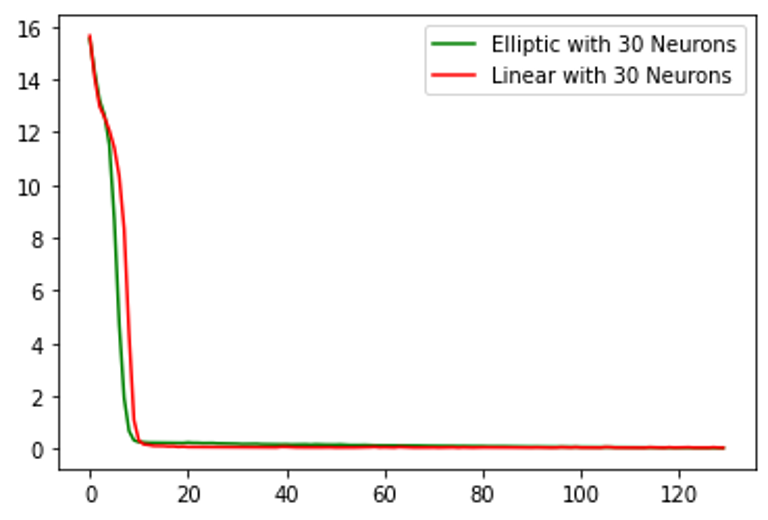} }}%

    \caption{A 3 hidden layer elliptic neural network converges faster than a 4 hidden layer affine linear one. In the $x-$axis it shows the number of epochs of the training, in the $y-$axis the mean squared error of the resulting approximation.}%
    \label{fig:deep_network1}%
\end{figure}

\begin{figure}[H]%
    \protect
    \centering
        \subfloat{{\includegraphics[width=7cm]{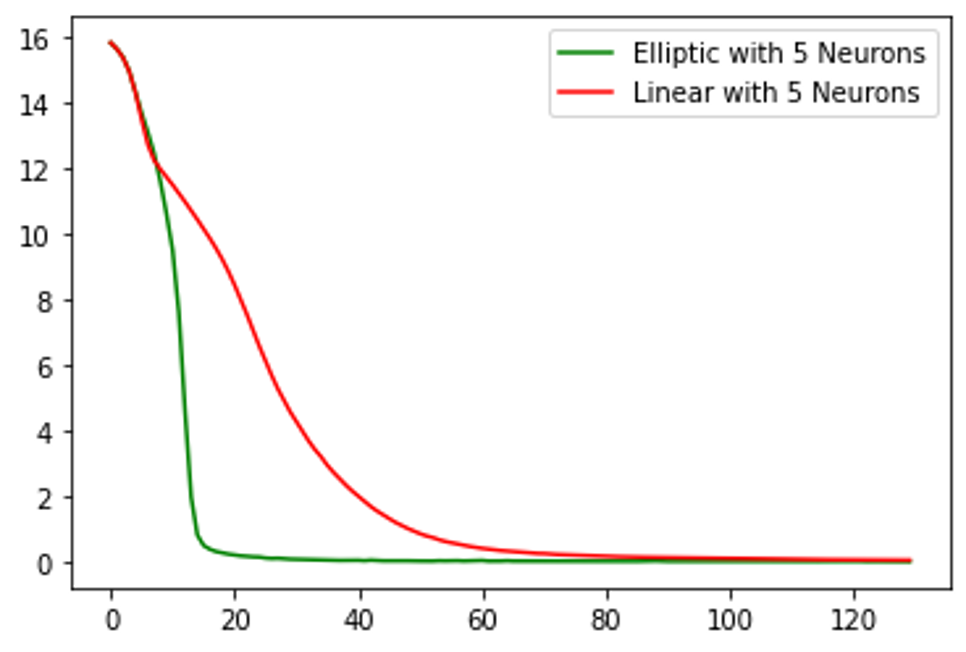} }}%

    \caption{A 4 hidden layer elliptic neural network converges faster than a 4 hidden layer affine linear one. Each hidden layer has $5$ neurons included, here we compare a 4-layer elliptic network with a 4-layer affine linear one.
    As before, it shows in the $x-$axis the number of epochs of the training, in the $y-$axis the mean squared error of the resulting approximation.}%
    \label{fig:deep_network2}%
\end{figure}

\begin{table}
\begin{tabular}{|c|c|c|c|c|c|}
 \hline
 Type of network
 &
 Training epochs
 &
 Hidden layers
 &
 Units per hidden layer
 &
 Test MSE
 &
 Test MAE\\
 \hline \hline
 \multirow{3}{4em}{Elliptic} & 140 & 3 & [5, 5, 5] & 0.0106&0.1752 \\
 &   140  & 3   & [30, 30, 30] & 0.0014&0.0404 \\
 & 140 & 4 &  [5, 5, 5, 5] & 0.0528 &0.0767 \\
 \hline
 \multirow{2}{6em}{Affine Linear} & 140 & 4 &  [5,5,5,5]&0.0105 &0.2074 \\
 &   140  & 4 & [30,30,30,30] &0.0453 &0.0618 \\
 \hline
\end{tabular}

\caption{Overview of the result of the numerical experiments for deep networks.}
 \label{table:deep3}%
\end{table}

\subsection{Clusterization examples}

In this subsection, we are going to look at a more "applied setting", in detail classification problems.

We will start with the well-known MNIST dataset (see \cite{Deng12}), which is a database consisting out of 60,000 examples of training data as well as 10,000 examples as test data of handwritten digits. The images which were used are size-normalized and centered in a fixed-size image.

Out of these images a t-distributed stochastic neighbour embedding (t-SNE) was generated. This procedure is a machine learning algorithm for dimensionality reduction and often used for visualization purposes. t-SNE preserves local structures of the dataset by letting the distances between points stay the same.

The use case for the quadratic neural networks would be the following: By only having the different clusters, the goal would now be to correctly classify MNIST images of which only the t-SNE embedding is known.

\subsubsection{Generation of training data and initialization}\label{sec:simulations}
The training data is generated via the t-distributed stochastic neighbour of the 60,000 examples - in the following one can see a visualization of it. One sees the clear groups of the different digits. For the test data set, we have generated $10,000$ data points, where the correct classification shall be determined. For means of simplicity, this paper observes the correct clustering of e.g. digit 8.

To determine the correct clustering for the other digits as well, one simply generates additional networks for the classification of the other numbers. When having conducted these experiments, this has increased the performance for the other digits as well. For the generation of the neighbour clustering, we have used code available on Kaggle for the t-sne-visualization.

\begin{figure}[H]

\centering

    \includegraphics[width=0.5\textwidth]{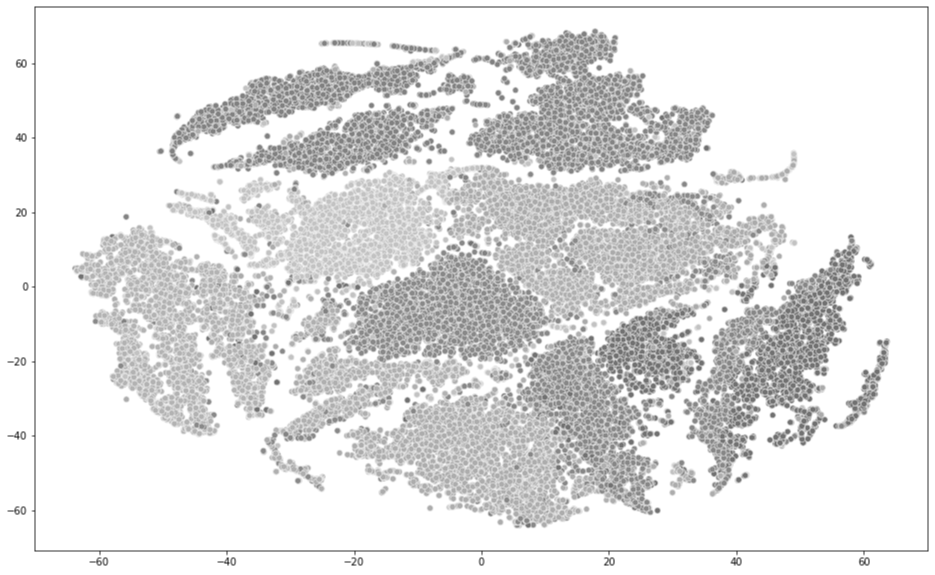}
\caption{t-distributed stochastic neighbour clustering}

\end{figure}

\subsubsection{Results}

The following table indicates the advantage of using elliptic (generalization of circular) decision functions and not linear ones. When using more epochs, hyperbolic and affine linear catches up with elliptic and have a similar approximation error (see also \autoref{table:deep}).

We have used different combinations to observe the different outcomes (see \autoref{fig:fig1_clustering}).

\begin{table}[H]
\begin{tabular}{|c|c|c|c|c|c|}
 \hline
 Type
 &
 Epochs
 &
 H. layers
 &
 Units per hidden layer
 &
 S. C. Crossentropy\\
 \hline \hline
Elliptic & 10 & 1 & [5] & 0.9515 \\
Elliptic  &   10  & 3   & [5,5,5] & 0.9757 \\
Elliptic   &   30  & 3   & [20,20,20] & 0.9770 \\
 \hline
Affine Linear & 10 & 1 & [5] & 0.9033 \\
Affine Linear &   10  & 3   & [5,5,5] & 0.9570 \\
Affine Linear  &   30  & 3   & [20,20,20] & 0.973 \\
 \hline
\end{tabular}
\caption{Overview of the result of the numerical experiments for deep networks.}
 \label{table:deep}%
\end{table}

\begin{figure}[H]%
    \protect
    \captionsetup[subfigure]{justification=centering}
    \centering
        \subfloat[{Classification via a 3-layer elliptic \\ network and 10 training epochs}][{Classification via a 3-layer elliptic \\ network and 10 training epochs}]{\includegraphics[width=7cm]{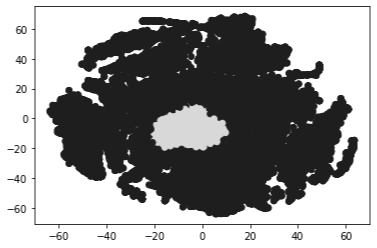} } %
    \qquad
        \subfloat[Classification via a 3-layer linear \\ network and 10 training epochs][Classification via a 3-layer linear \\ network and 10 training epochs]{{\includegraphics[width=7cm]{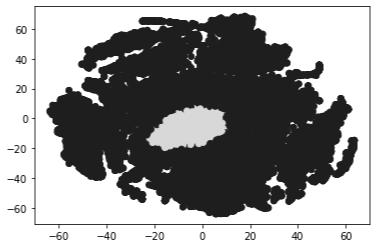} }}%

    \caption{Comparison of elliptic with linear affine layers on MNIST dataset}%
    \label{fig:fig1_clustering}%
\end{figure}

\subsubsection{Summary and possible extensions}

Having also performed other experiments, we can clearly state that there are several cases, under which the use of elliptical layers practically makes sense. Especially, when one is interested into convergence rates, one is obliged to use a small neural network structure (or even shallow neural networks). In such cases, these specially constructed decision functions clearly have an advantage.

\begin{figure}[H]%
    \protect
    \captionsetup[subfigure]{justification=centering}
    \centering
        \subfloat[{Image 1 where elliptical decision \\ functions have an advantage over linear}][{Image 1 where elliptical decision \\ functions have an advantage over linear}]{\includegraphics[width=7cm]{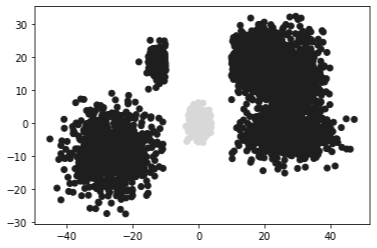} } %
    \qquad
        \subfloat[Image 2 where elliptical decision \\ functions have an advantage over linear][Image 2 where elliptical decision \\ functions have an advantage over linear]{{\includegraphics[width=7cm]{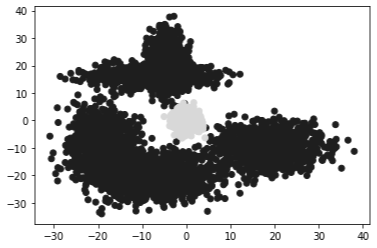} }}%

    \caption{Goal is to correctly identify the yellow subspecies out of a bigger population}%
    \label{fig:fig_subspecies}%
\end{figure}

In the following analysis, we will discuss the classification on \autoref{fig:fig_subspecies} (B) when one sole layer (shallow neural network structure) is used. Via those two "relatively simple" examples, it lets us understand which effect the different layers particularly have from a geometrical point of view.  This we will compare with the usage of e.g. classical linear layers. The setting was done similarly as above, we have used for the experiments a shallow neural network structure with 20 epochs of training. \\

We remember: The subpopulation should be identified out of a much bigger population. As already performed in the previous experiment, one has specific data points given in both the populations.   \\

One sees that the classification via an elliptical layer leads to a loss of 0.004 and an accuracy of 0.998. Therefore this structure can very well classify elliptical data points.

\begin{figure}[H]
  \centering
  \begin{minipage}[b]{0.45\textwidth}
    \includegraphics[width=\textwidth]{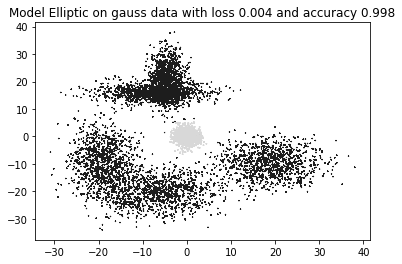}
    \caption{Classification via one elliptical layer and 20 epochs of training}
  \end{minipage}

\end{figure}

When one e.g. observes the classification via one parabolic layer, the following phenomena occurs:

\begin{figure}[H]
  \centering
  \begin{minipage}[b]{0.45\textwidth}
    \includegraphics[width=\textwidth]{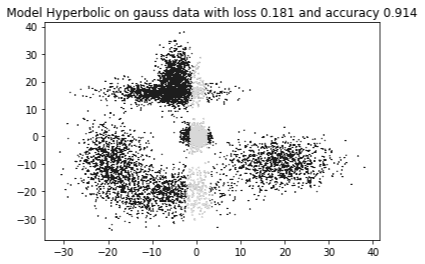}
    \caption{Classification via one hyperbolic layer and 20 epochs of training}
  \end{minipage}

\end{figure}

This is why it makes definite sense to choose a specific layer for a specific application case. This also comes into play when one uses a shallow linear network, which does not lead to reasonably well results.

\begin{figure}[H]
  \centering
  \begin{minipage}[b]{0.45\textwidth}
    \includegraphics[width=\textwidth]{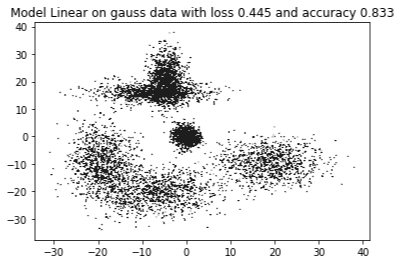}
    \caption{Classification via one linear layer and 20 epochs of training}
  \end{minipage}

\end{figure}

When using more epochs and more hidden layers, hyperbolic and affine linear catches up with elliptic and have a similar approximation error. A similar result has been observed in the previous experiment. \\

 \autoref{im:elliptic_acc} provides us with the accuracy of the elliptic model and  \autoref{im:elliptic_loss} the development of the loss function.

 \begin{figure}[H]
  \centering
  \begin{minipage}[b]{0.45\textwidth}
    \includegraphics[width=\textwidth]{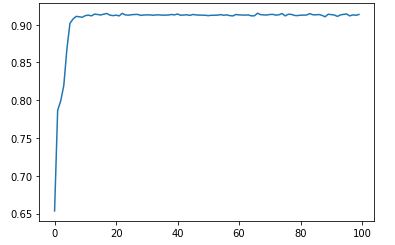}
    \caption{Development of the accuracy when using one shallow affine linear layer and increasing the number of epochs} \label{im:elliptic_acc}
  \end{minipage}

\end{figure}

 \begin{figure}[H]
  \centering
  \begin{minipage}[b]{0.45\textwidth}
    \includegraphics[width=\textwidth]{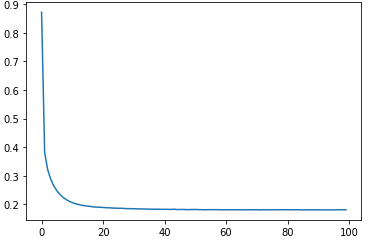}
    \caption{Development of the loss when using one shallow affine linear layer and increasing the number of epochs} \label{im:elliptic_loss}
  \end{minipage}

\end{figure}

\section{Conclusion}\label{sec:conclusions}
In this paper we suggested the use of new types of neurons in fully connected neural networks.
We proved a universal approximation theorem, which is applicable to such novel networks. Furthermore, we give the explicit convergence rates for the case of circular neurons, which has the same order as the classical affine linear neurons, but with vastly reduced number of neurons (in particular is spares one layer).

Additionally, the numerical results have confirmed the improved convergence for quadratic neural network functions when compared with affine linear ones, not only for three-layer networks, but also for deep neural networks. Potential next steps involve similar analysis for higher order polynomials, as only quadratic neural network functions have been covered  in this paper. Furthermore, higher dimension problems are a potential matter of interest in a follow-up paper.

\appendix
\section{Approximation to the identity (AtI)} \label{sec:appendix}
\begin{definition}[Approximation to the identity \cite{DenHan09}] \label{def:wavelet}
A sequence of {\bf symmetric kernel functions} $(S_k: \R^n \times \R^n  \to \R)_{k\in \Z}$ is said to be an {\bf approximation to the identity (AtI)} if there exist a quintuple $(\epsilon,\zeta,C,C_\rho,C_A)$ of positive numbers satisfying the additional constraints
\begin{equation} \label{eq:quin}
	0 < \epsilon \leq \frac{1}{n}, 0 < \zeta \leq \frac{1}{n} \text{ and } C_A < 1
\end{equation}
the following three conditions are satisfied for all $k \in \Z$:
\begin{enumerate}
\item \label{it1}
      $\abs{S_k(\vx,\vy)} \leq C\frac{2^{-k\epsilon}}{\left(2^{-k}+ C_\rho\norm{\vx-\vy}^n \right)^{1+\epsilon}}$ for all $\vx,\vy \in \R^n$;
\item \label{it2}
      $\abs{S_k(\vx,\vy) - S_k(\vx',\vy)} \leq C \left( \frac{C_\rho\norm{\vx-\vx'}^n }{2^{-k}+ C_\rho\norm{\vx-\vy}^n }\right)^{\zeta}\frac{2^{-k\epsilon}}{\left(2^{-k}+ C_\rho\norm{\vx-\vy}^n \right)^{1+\epsilon}}$ \\
      for all triples $(\vx,\vx',\vy) \in \R^n \times \R^n \times \R^n$ which satisfy
      \begin{equation} \label{eq:rest_item2}
      C_\rho\norm{\vx-\vx'}^n  \leq C_A \left(2^{-k}+ C_\rho\norm{\vx-\vy}^n \right);
      \end{equation}
\item \label{it4} $\int_{\R^n} S_k(\vx,\vy) d\vy = 1$ for all $\vx \in \R^n$.
\end{enumerate}
Moreover, we say that the AtI satisfies the {\bf double Lipschitz condition} if there exist a triple $(\tilde{C},\tilde{C}_A,\zeta)$ of positive constants satisfying
\begin{equation} \label{eq:quin2}
  \tilde{C}_A < \frac12,
\end{equation}
such that for all $k \in \Z$
      \begin{equation}\label{eq:DoubleLipschitzCondition}
      \begin{aligned}
      &\abs{S_k(\vx,\vy) - S_k(\vx',\vy) - S_k(\vx,\vy') + S_k(\vx',\vy')}\\
      \leq & \tilde{C} \left( \frac{C_\rho\norm{\vx-\vx'}^n }{2^{-k}+C_\rho\norm{\vx-\vy}^n} \right)^\zeta
      \left( \frac{C_\rho \norm{\vy-\vy'}^n}{2^{-k}+C_\rho\norm{\vx-\vy}^n }\right)^\zeta \frac{2^{-k \epsilon}}{(2^{-k}+C_\rho\norm{\vx-\vy}^n )^{1+\epsilon}}
\end{aligned}
      \end{equation}
      for all quadruples $(\vx,\vx',\vy,\vy')\in \R^n \times \R^n \times \R^n \times \R^n$ which satisfy
      \begin{equation} \label{eq:rest_item3}
  C_\rho \max\set{\norm{\vx-\vx'}^n,\norm{\vy-\vy'}^n} \leq \tilde{C}_A \left( 2^{-k}+C_\rho\norm{\vx-\vy}^n\right).
      \end{equation}
\end{definition}

The conditions \autoref{it2} and \autoref{eq:rest_item3} are essential for our analysis. We characterize now geometric properties of these constrained sets:
\begin{lemma} \label{lem:kk}
	\begin{itemize} Let $C_\rho=1$, $C_A=2^{-n}$ and $\tilde{C}_A=3^{-n}$.
		\item Then set of triples $(\vx,\vx',\vy)$ which satisfy \autoref{eq:rest_item2} and for which $\norm{\vx - \vy}^n \geq 2^{-k}$ and all $t \in[0,1]$ satisfy
		\begin{equation} \label{eq:kk}
             \norm{\vx+t(\vx'-\vx)-\vy}^n \geq 2^{-n}\norm{\vx-\vy}^n - 2^{-n} 2^{-k}.
		\end{equation}
	    \item
        The set of quadrupels $(\vx,\vx',\vy,\vy')$ which satisfy \autoref{eq:rest_item3}  and for which $\norm{\vx - \vy}^n \geq 2^{-k}$ satisfy
	    \begin{equation} \label{eq:kk1}	
         \norm{\vx+t_\vx(\vx'-\vx)-\vy-t_\vy(\vy'-\vy)}^n\geq 3^{-n} \norm{\vx-\vy}^n - 3^{-n} 2^{1-k}
	    \end{equation}
        for all $t_\vx,t_\vy\in[0,1]$.
	\end{itemize}
\end{lemma}
\begin{proof}
	\begin{itemize}
		\item With the concrete choice of parameters $C_A$, $C_\rho$ \autoref{eq:rest_item2} reads as follows
		\begin{equation} \label{eq:diff}
		 \norm{\vx-\vx'}^n \leq  2^{-n}\left(2^{-k} + \norm{\vx-\vy}^n\right).
		\end{equation}
        Since we assume that $\norm{\vx-\vy}^n \geq 2^{-k}$ it follows from \autoref{eq:diff} that
		\begin{equation*}
			\norm{\vx-\vx'} \leq  2^{-1}\norm{\vx-\vy}.
		\end{equation*}
        In particular $\norm{\vx-\vy}-\norm{\vx-\vx'} \geq 0$.

        We apply Jensen's inequality, which states that for $a,b \geq 0$
        \begin{equation} \label{eq:jensen}
          a^n+b^n \geq 2^{1-n}(a+b)^n.
        \end{equation}
        We use $a=\norm{\vx+t(\vx'-\vx)-\vy}$ and $b = \norm{t(\vx'-\vx)}$, which then (along with the triangle inequality) gives
        \begin{equation*}
           \norm{\vx+t(\vx'-\vx)-\vy}^n + \norm{t(\vx'-\vx)}^n \geq 2^{1-n}\left(\norm{\vx-\vy}\right)^n.
        \end{equation*}
        In other words, it follows from \autoref{eq:diff} that
        \begin{equation*}
           \begin{aligned}
           \norm{\vx+t(\vx'-\vx)-\vy}^n & \geq 2^{1-n}\norm{\vx-\vy}^n -  t^n\norm{\vx'-\vx}^n \\
           & \geq 2^{1-n}\norm{\vx-\vy}^n -  \norm{\vx'-\vx}^n \\
           & \geq 2^{1-n}\norm{\vx-\vy}^n - 2^{-n}\left(2^{-k} + \norm{\vx-\vy}^n\right)\\
           &= 2^{-n} \norm{\vx-\vy}^n - 2^{-n} 2^{-k}.
           \end{aligned}
        \end{equation*}
		\item With the concrete choice of parameters $\tilde{C}_A$, $C_\rho$ \autoref{eq:rest_item3} reads as follows
		\begin{equation} \label{eq:diffb}
		  \max \set{\norm{\vx-\vx'}^n,\norm{\vy-\vy'}^n} \leq  3^{-n} \left(2^{-k} + \norm{\vx-\vy}^n\right).
		\end{equation}
        Since we assume that $\norm{\vx-\vy}^n \geq 2^{-k}$ it follows from \autoref{eq:diffb} that
		\begin{equation*}
			\max \set{ \norm{\vx-\vx'}, \norm{\vy-\vy'}}  \leq  3^{-1} \norm{\vx-\vy}.
		\end{equation*}
        This in particular shows that
        \begin{equation*}
          \norm{\vx-\vy}-\norm{\vx-\vx'}-\norm{\vy-\vy'} \geq 0.
        \end{equation*}
        We apply Jensen's inequality, which states that for $a,b,c \geq 0$
        \begin{equation} \label{eq:jensen3}
          a^n+b^n+c^n \geq 3^{1-n}(a+b+c)^n.
        \end{equation}
        We use $a=\norm{\vx+t_\vx(\vx'-\vx)-\vy-t_\vy(\vy'-\vy)}$, $b = \norm{t_\vx(\vx'-\vx)}$ and $c = \norm{t_\vy(\vy'-\vy)}$, which then (along with the triangle inequality) gives
        \begin{equation*}
           \norm{\vx+t_\vx(\vx'-\vx)-\vy-t_\vy(\vy'-\vy)}^n + \norm{t_\vx(\vx'-\vx)}^n +\norm{t_\vy(\vy'-\vy)}^n \geq 3^{1-n}\norm{\vx-\vy}^n.
        \end{equation*}
        In other words, it follows from \autoref{eq:diffb} that
        \begin{equation*}
           \begin{aligned}
           \norm{\vx+t_\vx(\vx'-\vx)-\vy-t_\vy(\vy'-\vy)}^n & \geq 3^{1-n}\norm{\vx-\vy}^n -  t_\vx^n\norm{\vx'-\vx}^n - t_\vy^n\norm{\vy'-\vy}^n\\
           & \geq 3^{1-n}\norm{\vx-\vy}^n - \norm{\vx'-\vx}^n - \norm{\vy'-\vy}^n\\
           & \geq 3^{1-n}\norm{\vx-\vy}^n - 3^{-n} 2\left(2^{-k} + \norm{\vx-\vy}^n\right)\\
           &= 3^{-n} \norm{\vx-\vy}^n - 3^{-n} 2^{1-k}.
           \end{aligned}
        \end{equation*}
	\end{itemize}
	\end{proof}

The approximation to the identity in \autoref{def:wavelet} can be used to construct \emph{wavelet frames} that can approximate arbitrary functions in $\mathcal{L}^1(\R^n)$, as shown in the theorem below.
\begin{remark}
            When $\norm{\vx - \vy}^n < 2^{-k}$, \autoref{eq:kk} also holds since the right hand side of \autoref{eq:kk} is negative, so this inequality is trivial.
            \end{remark}

\begin{theorem}[\cite{ShaCloCoi18}] \label{le:WaveletApprox} Let $(S_k:\R^n \times \R^n \to \R)_{k \in \Z}$ be a symmetric AtI which
	satisfies the double Lipschitz condition (see \autoref{eq:DoubleLipschitzCondition}).
Let
\begin{equation} \label{eq:frame}
\psi_{k,\vy}(\vx) := 2^{-\frac{k}{2}}\left(S_k(\vx,\vy)-S_{k-1}(\vx,\vy)\right)
\text{ for all } \vx,\vy \in \R^n  \text{ and } k \in \Z.
\end{equation}
The the set of functions
\begin{equation} \label{eq:waveletframe}
\mathcal{F} := \set{ \vx \to \psi_{k,b}(\vx): k \in \Z , b \in 2^{-\frac{k}{d}}\Z },
\end{equation}
is a frame and for every function $f \in \mathcal{L}^1(\R^n)$ there exists a linear combination of $N$ elements of $\mathcal{F}$, denoted by $f_N$, satisfying
\begin{equation*}
\norm{f-f_N}_{L^2} \leq \norm{f}_{\mathcal{L}^1} (N+1)^{-1/2}.
\end{equation*}
\end{theorem}

\subsection*{Acknowledgements}
OS is supported by the Austrian Science Fund (FWF), with SFB F68, project F6807-N36 - Tomography with Uncertainties. LF is supported by the Austrian Science Fund (FWF) with SFB F68, project F6807-N36 (Tomography with Uncertainties). CS is supported by the Austrian Science Fund (FWF), with project AT0116011 (Photoacoustic tomography: analysis and numerics).
The financial support by the Austrian Federal Ministry for Digital and Economic
Affairs, the National Foundation for Research, Technology and Development and the Christian Doppler
Research Association is gratefully acknowledged.
\section*{References}
\renewcommand{\i}{\ii}
\printbibliography[heading=none]

\end{document}